\documentclass[reqno, a4paper]{amsart}

\usepackage{fixltx2e}
\usepackage[T1]{fontenc}
\usepackage[british]{babel}
\usepackage{amsfonts, amssymb, amsthm, amsmath}
\usepackage[all]{xy}
\usepackage{hyperref}
\usepackage{mathbbol}
\usepackage{calrsfs}
\usepackage{mathabx}
\usepackage{graphicx}

\theoremstyle{plain}

\newtheorem{lemma}[subsection]{Lemma}
\newtheorem{theorem}[subsection]{Theorem}
\newtheorem{proposition}[subsection]{Proposition}

\theoremstyle{definition}
\newtheorem{example}[subsection]{Example}
\newtheorem{examples}[subsection]{Examples}
\newtheorem{definition}[subsection]{Definition}

\newcommand{\som}[2]{\bigl\langle\begin{smallmatrix} {#1} \\ {#2}\end{smallmatrix}\bigr\rangle}
\newcommand{\bigsom}[2]{\left\langle\begin{smallmatrix} {#1} \\ {#2}\end{smallmatrix}\right\rangle}
\newcommand{\somm}[3]{\left\langle\begin{smallmatrix} {#1} \\ {#2} \\ {#3}\end{smallmatrix}\right\rangle}

\newcommand{\comp}{\raisebox{0.2mm}{\ensuremath{\scriptstyle{\circ}}}}
\newcommand{\defn}{\textbf}
\newcommand{\meet}{\wedge}
\newcommand{\join}{\vee}

\newcommand{\cosmash}{\diamond}
\newcommand{\normal}{\ensuremath{\lhd}}
\newcommand{\To}{\ensuremath{\Rightarrow}}

\renewcommand{\Im}{\ensuremath{\mathrm{Im}}}
\newcommand{\Eq}{\ensuremath{\mathrm{Eq}}}
\renewcommand{\ker}{\ensuremath{\mathrm{ker}}}
\newcommand{\Ker}{\ensuremath{\mathrm{Ker}}}

\newcommand{\Smith}{\ensuremath{\mathrm{S}}}

\newcommand{\X}{\ensuremath{\mathbb{X}}}

\newcommand{\Arr}{\ensuremath{\mathsf{Arr}}}

\newcommand{\Lie}{\ensuremath{\mathsf{Lie}}}
\newcommand{\Nil}{\ensuremath{\mathsf{Nil}}}
\newcommand{\DiGp}{\ensuremath{\mathsf{DiGp}}}
\newcommand{\Gp}{\ensuremath{\mathsf{Gp}}}
\newcommand{\Loop}{\ensuremath{\mathsf{Loop}}}
\newcommand{\Pt}{\ensuremath{\mathsf{Pt}}}
\newcommand{\HSLat}{\ensuremath{\mathsf{HSLat}}}
\newcommand{\DLat}{\ensuremath{\mathsf{DLat}}}

\renewcommand{\implies}{$\Rightarrow$}

\newcommand{\ACC}{{\rm (ACC)}}
\newcommand{\C}{{\rm (C)}}
\newcommand{\barC}{{\rm (\={C})}}

\newcommand{\FWACC}{{\rm (FWACC)}}
\newcommand{\LACC}{{\rm (LACC)}}

\newcommand{\SH}{{\rm (SH)}}
\newcommand{\SSH}{{\rm (SSH)}}

\newcommand{\W}{{\rm (W)}}

\newcommand{\noproof}{\hfill \qed}

\def\pullback{
 \ar@{-}[]+R+<6pt,-1pt>;[]+RD+<6pt,-6pt>%
 \ar@{-}[]+D+<1pt,-6pt>;[]+RD+<6pt,-6pt>}
 
\def\pushout{%
 \ar@{-}[]+L+<-6pt,1pt>;[]+LU+<-6pt,6pt>%
 \ar@{-}[]+U+<-1pt,6pt>;[]+LU+<-6pt,6pt>}

\def\splitpullback{%
 \ar@{-}[]+R+<6pt,-.01ex>;[]+RD+<6pt,-6pt>%
 \ar@{-}[]+D+<1.01ex,-6pt>;[]+RD+<6pt,-6pt>}

 \def\doublesplitpullback{%
 \ar@{-}[]+R+<6pt,-0.51ex>;[]+RD+<6pt,-6pt>%
 \ar@{-}[]+D+<0.51ex,-6pt>;[]+RD+<6pt,-6pt>}

\def\skewpullback{%
 \ar@{-}[]+LD+<-6pt,-6pt>;[]+LDD+<-6pt,-15.5pt>%
 \ar@{-}[]+D+<-1pt,-6pt>;[]+LDD+<-6pt,-15.5pt>}
 
\newdir{>}{{}*:(1,-.2)@^{>}*:(1,+.2)@_{>}}
\newdir{<}{{}*:(1,+.2)@^{<}*:(1,-.2)@_{<}}
\newdir{ |>}{{}*!/-3.5pt/@{|}*!/-8pt/:(1,-.2)@^{>}*!/-8pt/:(1,+.2)@_{>}}
\newdir{ >>}{{}*!/8pt/@{|}*!/3.5pt/:(1,-.2)@^{>}*!/3.5pt/:(1,+.2)@_{>}}
\newdir{ >}{{}*!/-8pt/@{>}}

\begin{document}

\date{\today}

\author{Nelson Martins-Ferreira}
\author{Tim Van~der Linden}

\email{martins.ferreira@ipleiria.pt}
\email{tim.vanderlinden@uclouvain.be}

\address{Departamento de Matem\'atica, Escola Superior de Tecnologia e Gest\~ao, Centro para o Desenvolvimento R\'apido e Sustentado do Produto, Instituto Poli\-t\'ecnico de Leiria, Leiria, Portugal}
\address{CMUC, Department of Mathematics, University of Coimbra, 3001--454 Coimbra, Portugal}
\address{Institut de Recherche en Math\'ematique et Physique, Universit\'e catholique de Louvain, chemin du cyclotron~2 bte~L7.01.02, B-1348 Louvain-la-Neuve, Belgium}

\thanks{The first author was supported by IPLeiria/ESTG-CDRSP and Funda\c c\~ao para a Ci\^encia e a Tecnologia (grants SFRH/BPD/4321/2008, PTDC/MAT/120222/2010 and PTDC/EME-CRO/120585/2010).}
\thanks{The second author works as \emph{chercheur qualifi\'e} for Fonds de la Recherche Scientifique--FNRS. His research was supported by Centro de Matem\'atica da Universidade de Coimbra and by Funda\c c\~ao para a Ci\^encia e a Tecnologia (grants SFRH/BPD/38797/2007 and PTDC/MAT/120222/2010). He wishes to thank CMUC and IPLeiria for their kind hospitality during his stays in Coimbra and in Leiria}



\title[Further remarks on the ``Smith is Huq'' condition]{Further remarks on the \\ ``Smith is Huq'' condition}

\begin{abstract}
We compare the \emph{Smith is Huq} condition~\SH\ with three commutator conditions in semi-abelian categories: first an apparently weaker condition which arose in joint work with Bourn and turns out to be equivalent with~\SH, then an apparently equivalent condition which takes commutation of non-normal subobjects into account and turns out to be stronger than~\SH. This leads to the even stronger condition that \emph{weighted commutators} in the sense of Gran, Janelidze and Ursini \emph{are independent of the chosen weight}, which is known to be false for groups but turns out to be true in any two-nilpotent semi-abelian category.
\end{abstract}

\maketitle

\section*{Introduction}

It is well known that, when they exist, the normalisations of a pair of Smith-commuting equivalence relations~\cite{Smith, Pedicchio} will always Huq-commute~\cite{Huq, BG} and that the converse need not hold: counterexamples exist in the category $\DiGp$ of digroups~\cite{Borceux-Bourn, Bourn2004} and in the category $\Loop$ of loops~\cite{HVdL}. A pointed Mal'tsev category satisfies the \defn{Smith is Huq} condition \defn{(SH)} if and only if two equivalence relations on a given object always centralise each other (= commute in the Smith sense) as soon as their normalisations commute.

The condition (SH) is fundamental in the study of internal categorical structures: it is shown in~\cite{MFVdL} that, for a semi-abelian category, this condition holds if and only if every star-multiplicative graph is an internal groupoid. As explained in~\cite{Janelidze} and in~\cite{HVdL} this is important when characterising internal crossed modules; furthermore, the condition has immediate (co)homological consequences~\cite{RVdL3}. 

Any pointed strongly protomodular exact category satisfies (SH)~\cite{BG} (in particular, so does any Moore category~\cite{Rodelo:Moore}) as well as any action accessible category~\cite{BJ07, AlanThesis} (in particular, any \emph{category of interest}~\cite{Montoli, Orzech}). Well-known concrete examples are the categories of groups, Lie algebras, associative algebras, non-unitary rings, and (pre)crossed modules of groups.

In the present paper we restrict ourselves to the context of semi-abelian categories~\cite{Janelidze-Marki-Tholen} and focus on three conditions which arose naturally in the study of \SH. In Section~\ref{Section barC} we prove that the condition
\begin{itemize}
\item[\barC] any change of base functor with respect to the fibration of points reflects the centralisation of equivalence relations
\end{itemize}
introduced in~\cite{BMFVdL} is equivalent to \SH\ as soon as the surrounding category $\X$ is semi-abelian (Theorem~\ref{Theorem barC implies C}). On the way we actually prove a stronger result, namely that under \SH\ also the change of base functors
\[
p^{*}\colon (\X\downarrow B)\to (\X\downarrow E)\qquad\text{for $p\colon E\to B$ in $\X$}
\]
with respect to the \defn{basic fibration} $\Arr(\X)\to \X$, which sends an arrow in $\X$ to its codomain, reflect the centralisation of equivalence relations.

Section~\ref{Section SH} recalls some definitions and (in Proposition~\ref{Proposition-SH}) gives an overview of conditions known to be equivalent to \SH. This is useful for the ensuing sections where those conditions are modified.

In Section~\ref{Section SSH} we explore a condition~\SSH\ which turns out to be strictly stronger than~\SH. Instead of asking that the change of base functors $p^{*}\colon{\Pt_{B}(\X)\to \Pt_{E}(\X)}$ with respect to the fibration of points $\P_{\X}\colon{\Pt(\X)\to \X}$ reflect the commutation of \emph{normal} subobjects, we shall assume that it reflects the commutation of \emph{all} subobjects, or, equivalently, of all cospans in~$\Pt_{B}(\X)$. Thus we obtain a kind of \emph{Smith is Huq} condition beyond the context where Smith commutators make sense, since a monomorphism can generally not be obtained as the normalisation of a relation. Proposition~\ref{Proposition-SSH} gives a list of properties, equivalent to \SSH. These happen to be implied by \emph{local algebraically cartesian closedness} \LACC~\cite{Bourn-Gray} so that, for instance, the category of groups satisfies it. On the other hand, we see that \SSH\ is strictly stronger than \SH, as the counterexample of Heyting semilattices (Example~\ref{Example-Lattices}) shows. This example also reveals that \emph{strong protomodularity}~\cite{Bourn2004} is not a sufficient condition for~\SSH.

The formulation of \SSH\ in terms of weighted commutators~\cite{GJU} given in Theorem~\ref{Theorem-weighted-SSH} is slightly subtle. This naturally leads to the final Section~\ref{Section W} where we look at a simplification. The resulting condition \W, which requires that \emph{weighted commutators are independent of the chosen weight}, turns out to be strictly stronger than \SSH, since even the category of groups does not satisfy it (\cite{GJU} and Example~\ref{Example Groups W}). Nevertheless any so-called \emph{two-nilpotent} semi-abelian category, such as for instance the category $\Nil_{2}(\Gp)$ of groups of nilpotency class $2$, does. The main result here is Proposition~\ref{Proposition-W} which gives equivalent conditions.

\section{A condition, equivalent to \SH}\label{Section barC}
In the article~\cite{BMFVdL}, it is shown that in a pointed Mal'tsev context, amongst the conditions of Theorem~\ref{Theorem barC implies C} below, (i) and (ii) are equivalent---this is the article's Theorem~2.1---while (ii) implies~(iii) (the article's Proposition~2.2). We show that when the surrounding category is semi-abelian, also the latter implication may be reversed.

\begin{theorem}\label{Theorem barC implies C}
In any semi-abelian category $\X$, the following conditions are equivalent:
\begin{enumerate}
\item the \emph{Smith is Huq} condition \SH;
\item condition~\C: any change of base functor with respect to the fibration of points reflects the commutation of normal subobjects;
\item condition~\barC: any change of base functor with respect to the fibration of points reflects the centralisation of equivalence relations;
\item any change of base functor with respect to the basic fibration reflects the centralisation of equivalence relations.
\end{enumerate}
\end{theorem}
\begin{proof}
To see that (iii) implies~(iv), consider equivalence relations~$R$ and~$S$ on an object $f\colon{X\to Y}$ of $(\X\downarrow Y)$ such that their pullbacks $y^{*}(R)$ and~$y^{*}(S)$ along an arrow $y\colon Y'\to Y$ centralise each other in $(\X\downarrow Y')$.
\[
\xymatrix{\Eq(f') \splitpullback \ar[r] \ar@<-1ex>[d]_-{f'_{1}} \ar@<1ex>[d]^-{f'_{0}} & \Eq(f) \ar@<-1ex>[d]_-{f_{1}} \ar@<1ex>[d]^-{f_{0}}\\
X' \pullback \ar[d]_-{f'} \ar[u] \ar[r]^-{x} & X \ar[u] \ar[d]^-{f} \\
Y' \ar[r]_-{y} & Y}
\]
We write $T$ and $U$ for the inverse images of $R$ and $S$ along the kernel pair projection~$f_{1}$, considered as equivalence relations in $\Pt_{X}(\X)$. Then their pullbacks~$x^{*}(T)$ and $x^{*}(U)$ coincide with the inverse images of $y^{*}(R)$ and~$y^{*}(S)$ along~$f_{1}'$. The equivalence relations $x^{*}(T)$ and $x^{*}(U)$ centralise each other in $\Pt_{X'}(\X)$ by left exactness of pullback functors and by the assumption that $y^{*}(R)$ and~$y^{*}(S)$ centralise each other. Condition~(iii) now implies that $T$ and $U$ centralise each other in~$\Pt_{X}(\X)$. As a consequence of \cite[Proposition~2.6.15]{Borceux-Bourn}, the regular images~$R=f_{1}(T)$ and~$S=f_{1}(U)$ centralise each other in~$\X$ and thus also in~$(\X\downarrow Y)$.

We now prove the implication (iv)~\implies~(i). Let $R$ and $S$ be equivalence relations on an object $X$ of $\X$ and $K$, $L\normal X$ their respective normalisations in $\X$. Suppose that~$K$ and~$L$ commute. Let $Y$ be the quotient $X/(K\join L)$ of $X$ by the join $K\join L$ of~$K$ and $L$. Then the quotient map $f\colon {X\to Y}$ makes $R$ and $S$ equivalence relations in~${(\X\downarrow Y)}$. Let $!_{Y}\colon {0\to Y}$ denote the unique morphism.
\[
\xymatrix{K\join L \pullback \ar[d] \ar@{{ |>}->}[r] & X \ar@{-{ >>}}[d]^-{f} \\
0 \ar[r]_-{!_{Y}} & Y}
\]
Since $K$ and $L$ still commute in $K\vee L=\Ker(f)=!_{Y}^{*}(X)$, a result due to Tomas Everaert and Marino Gran (published as Proposition~4.6 in~\cite{EverVdLRCT}) implies that the equivalence relations~$!_{Y}^{*}(R)$ and $!_{Y}^{*}(S)$ centralise each other in the semi-abelian category ${(\X\downarrow 0)}$. Condition~(iv) now tells us that~$R$ and $S$ centralise each other in~$(\X\downarrow Y)$ and thus also in~$\X$. 
\end{proof}

Note that outside the Barr exact context, the current proof of (iv)~\implies~(i) fails, because there a join of normal monomorphisms need no longer be normal. Using a different technique, in his recent article~\cite{BournCahiers2013}, Bourn extends Theorem~\ref{Theorem barC implies C} to regular Mal'tsev categories.

\section{Further equivalent conditions}\label{Section SH}

Let us now recall some further concepts and equivalent conditions which will be used in the following sections.

\subsection{Huq-commuting arrows in a category of points}
When
\begin{equation}\label{cospan-points}
\vcenter{\xymatrix@!0@=5em{A \ar@<-.5ex>[rd]_(.6){f} \ar[r]^-{\alpha} & D \ar@<-.5ex>[d]_(.4){p}
 & C \ar@<.5ex>[dl]^(.6){g} \ar[l]_-{\gamma}\\
& B \ar@<-.5ex>[u]_(.6){\beta} \ar@<-.5ex>[ul]_(.4){r}
\ar@<.5ex>[ur]^(.4){s}}}
\end{equation}
is a cospan in $\Pt_{B}(\X)$, it is easily seen that $\alpha$ and $\gamma$ Huq-commute if and only if there exists a morphism $\varphi\colon{A\times_{B}C\to D}$ in $\X$ such that $\varphi \comp e_{1}=\alpha$ and $\varphi \comp e_{2}=\gamma$. Here $e_{1}=\langle1_A,s\comp f \rangle$ and $e_{2}=\langle r\comp g,1_C \rangle$ are the morphisms in the pullback 
\[
\xymatrix@!0@=5em{ A\times_{B}C \doublesplitpullback \ar@<-.5ex>[r]_(.75){\pi_2}
\ar@<.5ex>[d]^-{\pi_1} & C \ar@<-.5ex>[l]_-{e_2} \ar@<-.5ex>[d]_-{g}
 \\
A
 \ar@<.5ex>[r]^-{f} \ar@<.5ex>[u]^-{e_1}
& B
 \ar@<.5ex>[l]^-{r} \ar@<-.5ex>[u]_-{s}}
\]
induced by the sections $r$ and $s$, respectively. 

In other words, $\alpha$ and $\gamma$ commute in $\Pt_{B}(\X)$ precisely when the triple $(\alpha,\beta,\gamma)$ is \defn{admissible with respect to $(f,r,g,s)$} in the sense of~\cite{MFVdL4} and the first author's thesis~\cite{MF-PhD}. These data are usually pictured in the shape of a diagram 
\begin{equation}\label{adm}
\vcenter{\xymatrix@!0@=4em{A \ar@<.5ex>[r]^-{f} \ar[rd]_-{\alpha} & B
\ar@<.5ex>[l]^-{r}
\ar@<-.5ex>[r]_-{s}
\ar[d]^-{\beta} & C \ar@<-.5ex>[l]_-{g} \ar[ld]^-{\gamma}\\
& D}}
\end{equation}
where $f\comp r=1_B=g\comp s$ and $\alpha\comp r=\beta=\gamma\comp s$ .

\subsection{Weighted commutation}\label{weighted commutation}
Recall from~\cite{GJU} that a \defn{weighted cospan} $(x,y,w)$ in $\X$ is a diagram
\begin{equation}\label{weighted cospan}
\vcenter{\xymatrix@!0@=4em{& W \ar[d]^-{w} \\
X \ar[r]_-{x} & D & Y. \ar[l]^-{y}}
}\end{equation}
It was shown in~\cite{MFVdL4} that the morphisms $x$ and $y$ \defn{commute over $w$} in the sense of~\cite{GJU} if and only if the triple $\bigl(\som{w}{x},w,\som{w}{y}\bigr)$ is admissible with respect to the quadruple $\bigl(\som{1_{W}}{0},\iota_{W},\som{1_{W}}{0},\iota_{W}\bigr)$. Reformulating this in terms of points, we see that the cospan
\begin{equation}\label{cospan-weight}
\vcenter{\xymatrix@!0@=7em{ W+X \ar@<-.5ex>[rd]_(.7){\som{1_{W}}{0}} \ar[r]^-{\left\langle\begin{smallmatrix} 1_{W} & 0\\
w & x\end{smallmatrix}\right\rangle} & W\times D \ar@<-.5ex>[d]_(.3){\pi_{W}}
 & W+Y \ar@<.5ex>[dl]^(.7){\som{1_{W}}{0}} \ar[l]_-{\left\langle\begin{smallmatrix} 1_{W} & 0\\
w & y\end{smallmatrix}\right\rangle}\\
& W \ar@<-.5ex>[u]_(.7){\langle1_{W},0\rangle} \ar@<-.5ex>[ul]_(.3){\iota_{W}}
\ar@<.5ex>[ur]^(.3){\iota_{W}}}}
\end{equation}
Huq-commutes in $\Pt_{W}(\X)$ if and only if $x$ and $y$ commute over $w$.

Conversely, we may view admissibility in terms of weighted commutativity as follows: again by~\cite{MFVdL4}, given a diagram such as~\eqref{adm}, it will be admissible if and only if~${x=\alpha\comp \ker(f)}$ and $y=\gamma\comp \ker(g)$ commute over $w=\beta\colon {W=B\to D}$. 

\subsection{Higgins commutators} 
By Theorem~1 in~\cite{MFVdL4}, for \defn{proper} morphisms $x$ and~$y$---which means that their images are normal subobjects~\cite{Bourn2001}---this commutation of $x$ and $y$ over $w$ may be reformulated in terms of Higgins commutators as the condition that both commutators
\[
[\Im(x),\Im(y)] \qquad \text{and}\qquad[\Im(x),\Im(y),\Im(w)]
\]
are trivial. Let us explain how to read this by recalling the needed definitions from~\cite{Actions, MM-NC, HVdL}. 

If $k\colon K \to D$ and $l\colon L \to D$ are subobjects of an object $D$, then the \defn{(Higgins) commutator} ${[K,L]\leq D}$ is the image of the induced morphism
\[
\xymatrix@=2em{K\cosmash L \ar@{{ |>}->}[r]^-{\iota_{K,L}} & K+L \ar[r]^-{\bigsom{k}{l}} & D,}
\]
where
\[
K\cosmash L=\Ker\bigl(\left\langle\begin{smallmatrix} 1_{K} & 0 \\
0 & 1_{L}\end{smallmatrix}\right\rangle
\colon K+L\to K\times L\bigr).
\]
It is easily seen that $k$ and $l$ Huq-commute if and only if $[K,L]$ vanishes.

If also $m\colon M \to D$ is a subobject of $D$, then we may define the \defn{ternary commutator} ${[K,L,M]\leq D}$ as the image of the composite
\[
\xymatrix@=3em{K\cosmash L\cosmash M \ar@{{ |>}->}[r]^-{\iota_{K,L,M}} & K+L+M \ar[r]^-{\somm{k}{l}{m}} & D,}
\]
where $\iota_{K,L,M}$ is the kernel of the morphism
\[
\xymatrix@=8em{K+L+M \ar[r]^-{\left\langle\begin{smallmatrix} i_{K} & i_{K} & 0 \\
i_{L} & 0 & i_{L}\\
0 & i_{M} & i_{M}\end{smallmatrix}\right\rangle} & (K+L)\times (K+M) \times (L+M).}
\]
The objects $K\cosmash L$ and $K\cosmash L\cosmash M$ are \emph{co-smash products} in the sense of~\cite{Smash}.

A key result here is the decomposition formula for the Higgins commutator of a join of two subobjects~\cite{Actions, HVdL}:
\[
[K,L\join M]=[K,L]\join [K,M]\join[K,L,M].
\]
This result is used in~\cite{HVdL, MFVdL4} to prove that when $k$ and $l$ are normal monomorphisms, they commute over $m$ if and only if
\[
[K,L]=0=[K,L,M].
\]

\subsection{The equivalent conditions}\label{def proper cospan}
We shall say that a cospan $(x,y)$
\[
\xymatrix{X \ar[r]^-{x} & D & Y \ar[l]_-{y}} 
\]
is \defn{proper} when $x$ and $y$ are proper (= have normal images). The following proposition gives an overview of conditions which are known to be equivalent to~\SH\ and which will be used later on.

\begin{proposition}\label{Proposition-SH}
In any semi-abelian category $\X$, the following conditions are equivalent:
\begin{enumerate}
\item the \emph{Smith is Huq} condition \SH;
\item for every morphism $p\colon {E\to B}$ in $\X$, the pullback functor
\[
p^{*}\colon{\Pt_{B}(\X)\to \Pt_{E}(\X)}
\]
reflects Huq-commutativity of proper cospans;
\item for every object $B$ of $\X$, the kernel functor $\Ker\colon\Pt_{B}(\X)\to\X$ reflects Huq-commutativity of proper cospans;
\item for any diagram such as~\eqref{adm} above in which $\alpha$ and $\gamma$ are regular epimorphisms, if $\alpha\comp \ker(f)$ and $\gamma\comp \ker(g)$ commute in~$D$, then there exists a morphism $\varphi\colon{A\times_{B}C\to D}$ such that $\varphi \comp e_{1}=\alpha$ and $\varphi \comp e_{2}=\gamma$;
\item for any weighted cospan $(x,y,w)$ such that $(x,y)$ is proper, the morphisms $x$ and $y$ commute over $w$ as soon as they Huq-commute;
\item for any given proper cospan $(x,y)$, the property of commuting over $w$ is independent of the chosen weight $w$ making $(x,y,w)$ a weighted cospan;
\item for any weighted cospan $(x,y,w)$ such that $(x,y)$ is proper,
\[
[\Im(x),\Im(y)]=0 \qquad\To\qquad [\Im(x),\Im(y),\Im(w)]=0.
\]
\end{enumerate}
\end{proposition}
\begin{proof}
The equivalence between (i) and (ii) is part of Theorem~\ref{Theorem barC implies C}. Variations on the equivalences between (ii), (iii) and (iv) are dealt with in Proposition~\ref{Proposition-SSH} below, and those between (v), (vi) and (vii) in Theorem~\ref{Theorem-weighted-SSH}. The assumption in~(iv)  that $\alpha$ and $\gamma$ are regular epimorphisms is there to ensure that $\alpha\comp \ker(f)$ and $\gamma\comp \ker(g)$ are proper. We are left with recalling that also the equivalence between (i) and (vi) is known~\cite{GJU, MFVdL4}. 
\end{proof}

\section{The stronger condition \SSH}\label{Section SSH}
The characterisation of the \emph{Smith is Huq} condition in terms of the fibration of points obtained in~\cite{BMFVdL} and recalled above immediately leads to a stronger condition (which we shall denote \SSH, even though it does not involve Smith commutators) and to the question whether or not the conditions \SH\ and \SSH\ are equivalent to each other. Indeed, by Theorem~\ref{Theorem barC implies C} we may write
\begin{itemize}
\item[\SH] any change of base functor with respect to the fibration of points reflects the commutation of normal subobjects;
\item[\SSH] any change of base functor with respect to the fibration of points reflects commutation (of arbitrary pairs of arrows)
\end{itemize}
and clearly \SSH\ \implies~\SH. In this section we give an overview of conditions which characterise \SSH, together with some examples (of categories which satisfy \SSH) and a counterexample (showing that \SSH\ is strictly stronger than \SH). Since they need some more work,  the conditions involving weighted commutativity are treated only later on in the section, starting from~\ref{Subsection (W)}.

\begin{proposition}\label{Proposition-SSH}
In any semi-abelian category $\X$, the following conditions are equivalent:
\begin{enumerate}
\item for every morphism $p\colon {E\to B}$ in $\X$, the pullback functor
\[
p^{*}\colon{\Pt_{B}(\X)\to \Pt_{E}(\X)}
\]
reflects Huq-commutativity of arbitrary cospans;
\item for every object $B$ of $\X$, the kernel functor $\Ker\colon\Pt_{B}(\X)\to\X$ reflects Huq-commutativity of arbitrary cospans;
\item for any diagram such as~\eqref{adm} above, if $\alpha\comp \ker(f)$ and $\gamma\comp \ker(g)$ commute in~$D$, then there exists a morphism $\varphi\colon{A\times_{B}C\to D}$ such that $\varphi \comp e_{1}=\alpha$ and $\varphi \comp e_{2}=\gamma$;
\item condition {\rm (i)} or condition {\rm (ii)}, restricted to cospans of monomorphisms.
\end{enumerate}
\end{proposition}
\begin{proof}
Condition~(ii) is the special case of~(i) where $p$ is~$!_{B}\colon{0\to B}$. The following standard trick shows (ii)~\implies~(i). For any $p\colon{E\to B}$ we have the induced inverse image functors
\[
\xymatrix{\Pt_{B}(\X) \ar[r]^-{p^{*}} & \Pt_{E}(\X) \ar[r]^-{!_{E}^{*}} & \Pt_{0}(\X)\cong \X.}
\]
Clearly $!_{E}^{*}\comp p^{*}=!_{B}^{*}=\Ker$. By assumption, this functor reflects Huq-com\-mu\-ta\-ti\-vi\-ty. But the kernel functor $!_{E}^{*}$ also \emph{preserves} Huq-commutating pairs of morphisms, and these two properties together give us~(i).

Condition~(ii) is the special case of (iii) where in the diagram~\eqref{adm} induced by~\eqref{cospan-points}, the arrow $\beta$ has a retraction $p$. We only need to check that the morphism~$\varphi$ induced by~(iii) is a morphism of points, of which the domain is $\overline{p}=f\comp \pi_{A}=g\comp \pi_{B}\colon {A\times_{B}C\to B}$ with section $\overline{\beta}=e_{1}\comp r=e_{2}\comp s\colon{B\to A\times_{B}C}$. Now
\[
p\comp \varphi \comp e_{1}=p\comp\alpha=f=f\comp \pi_{A}\comp e_{1}=\overline{p}\comp e_{1}
\]
and, similarly, $p\comp \varphi \comp e_{2}=\overline{p}\comp e_{2}$, so that $\overline{p}=p\comp\varphi$. Furthermore, $\varphi\comp \overline{\beta}=\varphi\comp e_{1}\comp r=\alpha\comp r=\beta$.

For the converse, it suffices to rewrite Diagram~\eqref{adm} in the shape
\[
\vcenter{\xymatrix@!0@=5em{A \ar@<-.5ex>[rd]_(.7){f} \ar[r]^-{\langle\alpha,f\rangle} & D\times B \ar@<-.5ex>[d]_(.3){\pi_{B}}
 & C \ar@<.5ex>[dl]^(.7){g} \ar[l]_-{\langle\gamma,g\rangle}\\
& B \ar@<-.5ex>[u]_(.7){\langle\beta,1_{B}\rangle} \ar@<-.5ex>[ul]_(.3){r}
\ar@<.5ex>[ur]^(.3){s}}}
\]
and consider it as a cospan $(\langle\alpha,f\rangle, \langle\gamma,g\rangle)$ in $\Pt_{B}(\X)$. So condition~(iii) is an instance of the kernel functors reflecting Huq-commutativity in the case where the point which is the codomain of the cospan is a product.

Finally, we can restrict (i) and (ii) to cospans of monomorphisms, because two arrows commute if and only if their regular images do~\cite[Proposition~1.6.4]{Borceux-Bourn}.
\end{proof}

\begin{definition}
We say that $\X$ satisfies \defn{condition \SSH} when it satisfies the equivalent conditions of Proposition~\ref{Proposition-SSH}.
\end{definition}

It is clear that \SSH\ implies \SH, the difference between the two essentially being that the former works for \emph{all} subobjects where the latter uses \emph{normal} subobjects. To find examples of categories where \SSH\ holds we may rely on the following:

\subsection{Even stronger conditions}
It is shown in~\cite{Bourn-Gray} that if a morphism~$p$ in a protomodular category is \emph{algebraically exponentiable} (that is, the pullback functor $p^*\colon{\Pt_B(\X)\to \Pt_{E}(\X)}$ has a right adjoint) then~$p^*$ reflects commuting pairs of arrows. More precisely, this means that if a category is pointed protomodular and one of the conditions
\begin{enumerate}
\item $\X$ is \emph{algebraically cartesian closed} \ACC
\item $\X$ is \emph{fibre-wise algebraically cartesian closed} \FWACC
\item $\X$ is \emph{locally algebraically cartesian closed} \LACC 
\end{enumerate}
holds then, respectively,
\begin{enumerate}
\item $(B\to 0)^*$ for each~$B$
\item pullback functors along effective descent morphisms
\item all pullback functors
\end{enumerate}
reflect commuting pairs. Hence \LACC\ implies \SSH\ via~(i) in Proposition~\ref{Proposition-SSH}.

\begin{examples}
The categories $\Gp$ of groups and $\Lie_{R}$ of Lie algebras over a commutative ring $R$ are \LACC, hence satisfy \SSH. 
\end{examples}

Another class of examples will be treated in Section~\ref{Section W}. It is easy (and illustrative, we believe) to make the groups case explicit:

\begin{example}\label{Example-Groups}
We show that when $\X$ is the category $\Gp$ of groups, the normalisation functor does reflect admissibility for all
diagrams~\eqref{adm}. In what follows we use additive notation, also for non-abelian groups. Consider in $\Gp$ the diagram
\[
\vcenter{\xymatrix@!0@=4em{X \ar@<.5ex>[r]^-{k} & A \ar@<.5ex>@{.>}[l]^-{k'} \ar@<.5ex>[r]^-{f} \ar[rd]_-{\alpha} & B \ar@<.5ex>[l]^-{r} \ar@<-.5ex>[r]_-{s}
\ar[d]^-{\beta} & C \ar@<-.5ex>[l]_-{g} \ar[ld]^-{\gamma} \ar@<-.5ex>@{.>}[r]_-{l'} & Y \ar@<-.5ex>[l]_-{l} \\
&& D}}
\]
in which $k=\ker (f)$, $l=\ker(g)$ and where $k'$, $l'$ are the unique functions (not homomorphisms) with the property that $kk'=1_{A}-rf$ and $ll'=-sg+1_{C}$ (so that $k'k=1_{X}$ and $l'l=1_{Y}$). Note that
\[
a=kk'(a)+rf(a)\qquad\text{and}\qquad c=sg(c)+ll'(c)
\]
for all $a\in A$, $c\in C$.

Assuming that $\alpha k$ and $\gamma l$ commute, we have to construct a suitable group homomorphism $\varphi\colon{A\times_{B}C \to D}$ to show that $(\alpha,\beta,\gamma)$ are admissible. We define
\[
\varphi(a,c)=\alpha kk'(a)+\gamma(c)
\]
and prove that $\varphi(a+a',c+c')=\varphi(a,c)+\varphi(a',c')$. Note that
\begin{align*}
k'(a+a') &= 1_{A}(a+a') - rf(a+a')\\
&= kk'(a)+rf(a)+kk'(a')+rf(a')-rf(a+a')\\
&= \underbrace{kk'(a)}_{\in X}+\underbrace{rf(a)+kk'(a')-rf(a)}_{\in X}.
\end{align*}
Now for all $x\in X$, $b\in B$, we have that
\[
\alpha k(r(b)+k(x)-r(b))=\beta(b)+\alpha k(x)-\beta(b).
\]
Hence, on the one hand,
\begin{align*}
& \varphi(a+a',c+c') \\
&=\alpha kk'(a+a') + \gamma(c+c')\\
&= \alpha k (kk'(a) + rf(a) + kk'(a') - rf(a)) + \gamma(c) + \gamma(c')\\
&= \alpha kk'(a) + \alpha k(r(b)+kk'(a')-r(b))) + \gamma(sg(c)+ll'(c)) + \gamma(c')\\
&= \alpha kk'(a) + \beta (b) + \alpha kk'(a') - \beta (b)+ \beta(b) + \gamma ll'(c) + \gamma(c')\\
&= \alpha kk'(a) + \beta (b) + \alpha kk'(a') +\gamma ll'(c) + \gamma(c'),
\end{align*}
where $f(a)=b=g(c)$, while on the other hand
\begin{align*}
\varphi(a',c') + \varphi(a',c') &=\alpha kk'(a) + \gamma(c) + \alpha kk'(a') + \gamma(c')\\
&= \alpha kk'(a) + \beta(b) + \gamma ll'(c) + \alpha kk'(a') + \gamma(c').
\end{align*}
Since, by assumption, $\gamma ll'(c) + \alpha kk'(a')=\alpha kk'(a') +\gamma ll'(c)$, these two expressions are equal to each other, and $\varphi$ is a homomorphism.
\end{example}

\subsection{\SSH\ is strictly stronger than \SH}
We now consider the question whether or not \SSH\ and \SH\ are equivalent. It turns out that the answer is ``no'', via the following counterexample. We first show that all arithmetical categories satisfy \SH, then we prove that the arithmetical Moore category $\HSLat$ of Heyting semilattices does not satisfy \SSH.

\subsection{All arithmetical categories satisfy \SH}
An exact Mal'tsev category is called \defn{arithmetical}~\cite{Pedicchio2, Borceux-Bourn} when every internal groupoid is an equivalence relation. This implies that the Smith commutator $[R,S]^{\Smith}$ of two equivalence relations $R$ and $S$ on an object $D$ is their intersection $R\meet S$. It is also well known~\cite{Borceux-Bourn} that if the category is, moreover, pointed, then the only abelian object is the zero object, and (assuming that binary sums exist) the Higgins commutator $[X,Y]$ of normal subobjects $X$, $Y\normal D$ is the intersection~$X\meet Y$. Since the normalisation functor preserves intersections, it follows that any pointed arithmetical category with binary sums has the \emph{Smith is Huq} property.

Two examples of this situation which are relevant to us are the category $\HSLat$ of Heyting (meet) semi-lattices and $\DLat$ of distributive lattices. The latter is only weakly Mal'tsev~\cite{NMF1, NMF3}, but it is easily seen that it fits the above picture, hence (trivially) satisfies \emph{Smith is Huq}. On the other hand, the category $\HSLat$ is semi-abelian~\cite{Jo}---in fact it even is a Moore category~\cite{Rodelo:Moore}---and satisfies the conditions of Theorem~\ref{Theorem barC implies C}, but nevertheless does not satisfy \SSH. This tells us that the conditions of Proposition~\ref{Proposition-SSH} are strictly stronger than the \emph{Smith is Huq} property. Furthermore, they are not even implied by the \emph{strong protomodularity} condition~\cite{Bourn2004}, which by definition all Moore categories satisfy.

\begin{example}\label{Example-Lattices}
A concrete counterexample is the diagram~\eqref{adm} in $\HSLat$ defined as follows: $A=D=\{0,\tfrac{1}{2},1\}$ and $B=\{0,1\}$ with the natural order and implication
\[
p\Rightarrow q \qquad=\qquad\begin{cases}
q & \text{if $p>q$}\\
1 & \text{otherwise,}
\end{cases}
\]
and~$C$ is the boolean algebra
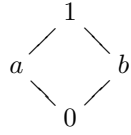
\[
\xymatrix@!0@=2em{& 1 \ar@{-}[ld] \ar@{-}[rd] \\
a && b\\
& 0 \ar@{-}[ul] \ar@{-}[ur]}
\]
with $\neg a=b$; the tables
\begin{center}
\begin{tabular}{c|ccc}
$A$ & $0$ & $\tfrac{1}{2}$ & $1$\\[.5ex]
\hline
 $f$ & $0$ & $1$ & $1$\\
$\alpha$ & $0$ & $\tfrac{1}{2}$ & $1$
\end{tabular}
\qquad\qquad
\begin{tabular}{c|cc}
$B$ & $0$ & $1$\\
\hline
 $r$ & $0$ & $1$ \\
$s$ & $0$ & $1$\\
$\beta$ & $0$ & $1$
\end{tabular}
\qquad\qquad
\begin{tabular}{c|cccc}
$C$ & $0$ & $a$ & $b$ & $1$\\
\hline
 $g$ & $0$ & $0$ & $1$ & $1$\\
$\gamma$ & $0$ & $1$ & $1$ & $1$
\end{tabular}
\end{center}
determine the morphisms between them. Let $X=\{\tfrac{1}{2},1\}$ and $Y=\{b,1\}$ be the kernels of $f$ and $g$, respectively. Then their direct images along $\alpha$ and~$\gamma$ become $\{\tfrac{1}{2},1\}$ and $\{1\}$, which Huq-commute in $D$. On the other hand, the triple $(\alpha,\beta,\gamma)$ is not admissible with respect to $(f,r,g,s)$. Indeed, as an ordered set, the pullback of $f$ and $g$ is given by the following diagram.
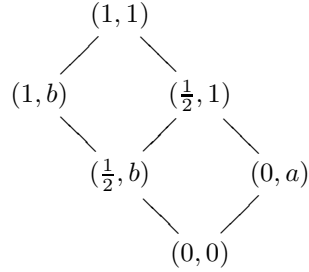
\[
\xymatrix@!0@=3em{& (1,1) \ar@{-}[ld] \ar@{-}[rd] \\
(1,b) && (\tfrac{1}{2},1) \ar@{-}[dr]\\
& (\tfrac{1}{2},b) \ar@{-}[ul] \ar@{-}[ur] && (0,a)\\
&& (0,0) \ar@{-}[ul] \ar@{-}[ur]}
\]
Hence if a function $\varphi$ as in the definition of admissibility exists, then necessarily
\[
\varphi(0,a)=\varphi(rg(a),a)=\gamma(a)=1
\]
and $\varphi(\tfrac{1}{2},1)=\varphi(\tfrac{1}{2},sf(\tfrac{1}{2}))=\alpha(\tfrac{1}{2})=\tfrac{1}{2}$; but this function cannot preserve the order.
\end{example}

\subsection{Conditions in terms of weighted commutators}\label{Subsection (W)}
Given $X\leq D$ and $w\colon{W\to D}$ a morphism, we say that $X$ is \defn{$w$-normal} when \defn{$\Im(w)$ normalises~$X$}, which means that $[\Im(w),X]\leq X$. It is explained in~\cite{MM-NC} that $X$ is normal if and only if it is $1_{D}$-normal. Proposition~4.11 in~\cite{Actions} further tells us that the following are equivalent:
\begin{enumerate} 
\item $\Im(w)$ normalises $X$;
\item $X$ is normal in the join $\Im(w)\join X$;
\item $\Im(w)\join X$ acts on $X$ by conjugation.
\end{enumerate}

We call $[\Im(w),X]\join X$ the \defn{$w$--normal closure} of $X$ in $D$. In particular~\cite{Actions}, the $1_{D}$--normal closure of $D$ is the ordinary normal closure, i.e., the smallest normal subobject of $D$ containing $X$. On the other hand, the $0$--normal closure of $x$ is $x$ itself. Freely using~\cite[Proposition~2.21]{HVdL}, we may see that in general 
\begin{align*}
&[\Im(w), [\Im(w),X]\join X]\\
=\,&[\Im(w), [\Im(w),X]]\join [\Im(w),X]\join [\Im(w),[\Im(w),X],X]\\
\leq\, &[\Im(w),X]\\
\leq\, &[\Im(w),X]\join X
\end{align*}
so that the $w$--normal closure of $X$ in $D$ is indeed $w$-normal, while its minimality is clear.

Given a weighted cospan $(x,y,w)$ as in~\eqref{weighted cospan}, we shall say that the cospan $(x,y)$ is \defn{$w$-proper} if and only if the images of $x$ and $y$ are $w$-normal. That is to say,
\[
[\Im(w),\Im(x)]\leq \Im(x)\qquad \text{and}\qquad[\Im(w),\Im(y)]\leq \Im(y).
\]
So a $1_{D}$-proper cospan $(x,y)$ is nothing but a proper cospan in the sense of~\ref{def proper cospan}, while any cospan is $0$-proper.

\begin{lemma}\label{Lemma Normal Closure}
Suppose that $(x,y,w)$ is a weighted cospan as in~\eqref{weighted cospan}. Then the image of the morphism $\som{w}{x}\comp\kappa_{W,X}\colon W\flat {X\to D}$ is the $w$--normal closure of $\Im(x)$ in $D$. In particular, $\som{w}{x}\comp\kappa_{W,X}$ factors over $\Im(x)$ if and only if $x$ is $w$-proper.
\[
\xymatrix{W\flat X \ar@{{ |>}->}[r]^-{\kappa_{W,X}} \ar@{.>}[d]_-{\xi} & W+X \ar[d]^-{\som{w}{x}} \\
\Im(x) \ar@{{ >}->}[r] & D}
\]
\end{lemma}
\begin{proof}
The inclusions in the split short exact sequence 
\[
\xymatrix@C=4em{0 \ar[r] & W\cosmash X \ar@{{ |>}->}[r]^-{\iota_{W,X}} & W\flat X \ar@<-.5ex>@{-{ >>}}[r]_-{\som{0}{1_{X}}\circ\kappa_{W,X}} & X \ar[r] \ar@<-.5ex>@{{ >}->}[l]_-{\eta_{X}^{W}} & 0}
\]
are jointly regular epic, which implies that 
\begin{equation*}
\Im(\som{w}{x}\comp\kappa_{W,X})=\Im(\som{w}{x}\comp\iota_{W,X})\join \Im(x)= [\Im(w),\Im(x)]\join \Im(x).
\end{equation*}
In other words, the unique dotted lifting in the diagram
\[
\xymatrix{(W\diamond X)+ X \ar@{-{ >>}}[r]^-{\som{\iota_{W,X}}{\eta_{X}^{W}}} \ar@{-{ >>}}[d] & W\flat X \ar@{.>}[ld] \ar[d]^-{\som{w}{x}\circ\kappa_{W,X}} \\
[\Im(w),\Im(x)]\join \Im(x) \ar@{{ >}->}[r] & D}
\]
is necessarily a regular epimorphism. The morphism $x$ being $w$-proper means precisely that the $w$--normal closure of $\Im(x)$ in $D$ is $\Im(x)$ itself.
\end{proof}

\begin{lemma}\label{Lemma w-Normal Cospan}
Suppose that $(x,y,w)$ is a weighted cospan as in~\eqref{weighted cospan}. If $(x,y)$ is a $w$-proper cospan, then
\[
\som{w}{x}\comp\kappa_{W,X}\colon W\flat {X\to D} \qquad\text{and}\qquad \som{w}{y}\comp\kappa_{W,Y}\colon W\flat {Y\to D}
\]
Huq-commute if and only if do $x$ and $y$.
\end{lemma}
\begin{proof}
By Lemma~\ref{Lemma Normal Closure}, these morphisms have $\Im(x)$ and $\Im(y)$ as their respective images. The result now follows form~\cite[Proposition~1.6.4]{Borceux-Bourn}.
\end{proof}

\begin{theorem}\label{Theorem-weighted-SSH}
In any semi-abelian category $\X$, the following are equivalent:
\begin{enumerate}
\item the condition~\SSH;
\item for any weighted cospan $(x,y,w)$ such that the cospan $(x,y)$ is $w$-proper, the morphisms $x$ and $y$ commute over $w$ as soon as they Huq-commute;
\item for any cospan $(x,y)$, the property of commuting over $w$ is independent of the chosen weight $w$ for which the cospan $(x,y)$ is $w$-proper;
\item for any weighted cospan $(x,y,w)$ such that the cospan $(x,y)$ is $w$-proper, if $[\Im(x),\Im(y)]=0$ then $[\Im(x),\Im(y),\Im(w)]=0$.
\end{enumerate}
\end{theorem}
\begin{proof}
Since $x$ and $y$ Huq-commute if and only if they commute over zero~\cite{GJU}, it is clear that (ii) and (iii) are equivalent. The equivalence between (iii) and (iv) is a direct consequence of~\cite[Theorem~2]{MFVdL4}.

(ii) \implies~(i) via Condition~(iii) in~Proposition~\ref{Proposition-SSH} as explained in Subsection~\ref{weighted commutation} above. To prove the converse, consider a weighted cospan $(x,y,w)$ such that the cospan $(x,y)$ is $w$-proper. The image through the kernel functor of the induced cospan~\eqref{cospan-weight} in $\Pt_{W}(\X)$ will commute if and only if $\som{w}{x}\comp \ker\som{1_{W}}{0}\colon {W\flat X\to D}$ Huq-commutes with $\som{w}{y}\comp \ker\som{1_{W}}{0}\colon {W\flat Y\to D}$ in $\X$. By Lemma~\ref{Lemma w-Normal Cospan}, this is equivalent to $x$ and $y$ Huq-commuting in $\X$. Hence (ii) follows from Condition~(ii) in~Proposition~\ref{Proposition-SSH}.
\end{proof}

Part of the above argument is worth repeating for later use:

\begin{lemma}\label{Lemma (x,y,w) commutes}
If $(x,y,w)$ is a weighted cospan as in~\eqref{weighted cospan} and $x$ and $y$ commute over $w$, then the $w$--normal closures of $\Im(x)$ and $\Im(y)$ in $D$ Huq-commute.
\end{lemma}
\begin{proof}
By Lemma~\ref{Lemma Normal Closure}, the $w$--normal closures in question are the respective images through the kernel functor of the induced cospan~\eqref{cospan-weight} in $\Pt_{W}(\X)$.
\end{proof}

\section{When is the weighted commutator independent\\ of the chosen weight?}\label{Section W}

The rather subtle condition on weighted cospans needed in Theorem~\ref{Theorem-weighted-SSH} naturally leads to the question when, for \emph{all} weighted cospans, weighted commutativity is independent of the chosen weight. It it known that this does not even happen in the category of groups (see Example~\ref{Example Groups W} and also~\cite{GJU} where further examples are given); it will, however, do happen in any so-called \emph{two-nilpotent} semi-abelian category. 

Essentially repeating the proof of Theorem~\ref{Theorem-weighted-SSH}, we find:

\begin{proposition}\label{Proposition-W}
In any semi-abelian category $\X$, the following conditions are equivalent:
\begin{enumerate}
\item for any weighted cospan $(x,y,w)$, the morphisms $x$ and $y$ commute over $w$ as soon as they Huq-commute;
\item for any given cospan $(x,y)$, the property of commuting over $w$ is independent of the chosen weight $w$ making $(x,y,w)$ a weighted cospan in $\X$;
\item for any given weighted cospan $(x,y,w)$ in $\X$, if $[\Im(x),\Im(y)]=0$ then $[\Im(x),\Im(y),\Im(w)]=0$;
\item any of the above, restricted to cospans of monomorphisms.\noproof
\end{enumerate}
\end{proposition}

\begin{definition}
We say that $\X$ satisfies \defn{condition \W} when it satisfies the equivalent conditions of Proposition~\ref{Proposition-W}.
\end{definition}

The category $\Gp$ of groups does not satisfy \W, as shows the following example.

\begin{example}\label{Example Groups W}
Consider the cyclic group of order two $C_{2}$ as a subgroup of the symmetric group on three elements $S_{3}$ via a monomorphism $x\colon{C_{2}\to S_{3}}$. Then $x$ commutes with $x$ over $0$ since $C_{2}$ is abelian. \W\ in the guise of Condition~(i) in Proposition~\ref{Proposition-W} would imply that $x$ commutes with $x$ over $1_{S_{3}}$, which gives us a contradiction with $S_{3}$ being non-abelian. Indeed, the normal closure of $x$ is~$1_{S_{3}}$ itself, so that we would have $[S_{3},S_{3}]=0$ by Lemma~\ref{Lemma (x,y,w) commutes}.
\end{example}

As an example of a non-abelian category satisfying \W\ we find the semi-abelian variety $\Nil_{2}(\Gp)$ of groups of nilpotency class $2$, in which \emph{all} ternary commutators vanish, because $[[X,X],X]$ and $[X,X,X]$ coincide in the category $\Gp$. This is clearly true in general once we adopt the following definition due to Manfred Hartl~\cite{CCC}. 

\begin{definition}
A semi-abelian category is called \defn{two-nilpotent} if all ternary co-smash products in it are trivial.
\end{definition}

This condition is equivalent to all ternary Higgins commutators vanishing. It is clear that such categories (trivially) satisfy (iii) in Proposition~\ref{Proposition-W}, so \W\ holds in all two-nilpotent categories. Examples include categories of modules over a square ring, and in particular categories of algebras over a nilpotent algebraic operad of class two~\cite{BHP}.

\section{Overview}

\noindent
\resizebox{\textwidth}{!}
{\begin{tabular}{cccccc}
\hline
\txt{condition}  & \txt{weighted commutator\\ independent of\\ chosen weight $w$ for} & \multicolumn{2}{c}{\txt{inverse image functors\\ reflect commutation of}} & \txt{counterexample} & \txt{holds for}\\
\hline\\
\W & all cospans & --- & --- & $\Gp$ & two-nilpotent categories\\\\
\SSH & $w$-proper cospans & \txt{all cospans} & --- & $\HSLat$ & \LACC\ categories\\\\
\SH & proper cospans & proper cospans & equivalence relations & $\Loop$, $\DiGp$ & \emph{categories of interest}\\\\
\hline
\end{tabular}}


\providecommand{\noopsort}[1]{}
\providecommand{\bysame}{\leavevmode\hbox to3em{\hrulefill}\thinspace}
\providecommand{\MR}{\relax\ifhmode\unskip\space\fi MR }
\providecommand{\MRhref}[2]{%
  \href{http://www.ams.org/mathscinet-getitem?mr=#1}{#2}
}
\providecommand{\href}[2]{#2}

\end{document}